\documentclass[a4paper,10pt,reqno]{amsart}

\usepackage{ifpdf}
\ifpdf 
    \usepackage[pdftex]{graphicx}   
    \usepackage[pdftex,     
            plainpages=false,   
            breaklinks=true,    
            colorlinks=true,
            pdftitle=My Document
            pdfauthor=My Good Self
           ]{hyperref} 
\else 
    \usepackage{graphicx}       
    \usepackage{hyperref}       
\fi 

\usepackage [latin1]{inputenc}
\usepackage{amsfonts,amsmath}	
\usepackage{amssymb}
\usepackage{verbatim}
\usepackage{amsopn}
\usepackage[english]{babel}
\usepackage{amsthm}
\usepackage{enumerate}
\usepackage{mathrsfs}	
\usepackage{mathtools}
\usepackage{color}
\usepackage{bbm}
\usepackage{esint}

\email{elio.marconi@unibas.ch}

\theoremstyle{definition} \newtheorem{definition}{Definition}[section]
\theoremstyle{definition} \newtheorem{remark}[definition]{Remark}
\theoremstyle{definition} \newtheorem{assum}[definition]{Assumption}
\theoremstyle{plain} \newtheorem{lemma}[definition]{Lemma}
\theoremstyle{plain} \newtheorem{proposition}[definition]{Proposition}
\theoremstyle{plain} \newtheorem{theorem}[definition]{Theorem}
\theoremstyle{plain} \newtheorem{corollary}[definition]{Corollary}
\theoremstyle{definition} 
\theoremstyle{plain} 
\theoremstyle{definition} \newtheorem{assumption}[definition]{Assumption}

\DeclareMathOperator{\supp}{supp}
\DeclareMathOperator{\Span}{span}

\newcommand{\R}{\mathbb{R}}

\newcommand{\N}{\mathbb{N}}
\newcommand{\Z}{\mathbb{Z}}
\newcommand{\TV}{\text{\rm Tot.Var.}}
\newcommand{\E}{\mathscr E}
\renewcommand{\H}{\mathscr H}
\newcommand{\M}{\mathcal M}
\newcommand{\D}{\mathcal D}
\newcommand{\Lip}{\mathrm{Lip}}
\newcommand{\Id}{\mathbb{I}}

\renewcommand{\div}{\mathrm{div}}
\newcommand{\loc}{\mathrm{loc}}

\DeclareMathOperator{\BV}{BV}

\newcommand{\FT}{\mathrm{FT}}
\newcommand{\e}{\varepsilon}

\numberwithin{equation}{section} 

\title[Solutions to scalar conservation laws with finite entropy production]{On the structure of weak solutions to scalar conservation laws with finite entropy production}

\author{Elio Marconi}

\thanks{The author acknowledges ERC Starting Grant 676675 FLIRT}

\begin{document}

\maketitle

\begin{abstract}
We consider weak solutions with finite entropy production to the scalar conservation law
\begin{equation*}
\partial_t u+\div_x F(u)=0 \quad \mbox{in }(0,T)\times \R^d.
\end{equation*}
Building on the kinetic formulation we prove under suitable nonlinearity assumption on $f$ that the set of non Lebesgue points of $u$ has 
Hausdorff dimension at most $d$.
A notion of Lagrangian representation for this class of solutions is introduced and this allows for a new interpretation of the entropy dissipation
measure.
\end{abstract}

\section{Introduction}
We study the structure of weak solutions to the scalar conservation law 
\begin{equation}\label{E_CL_intro}
\partial_t u + \div_x F(u)=0 \qquad \mbox{in }(0,T)\times \R^d.
\end{equation}
It is well-known that uniqueness for the Cauchy problem associated to \eqref{E_CL_intro} fails in the class of weak solutions and it is restored
by requiring the dissipation of every convex entropy:
\begin{equation}\label{E_mu_intro}
\mu_\eta:=\partial_t\eta(u)+\div_x Q(u) \le 0 \quad \mbox{in }\mathcal D',
\end{equation}
where $\eta''(v)\ge 0$ and $Q'(v)=\eta'(v)F'(v)$ for every $v\in \R$. Bounded solutions satisfying \eqref{E_mu_intro} for every convex entropy 
$\eta$ and
corresponding flux $Q$ are called entropy solutions and the Cauchy problem associated to \eqref{E_CL_intro} is well-posed in this class \cite{Kruzhkov_contraction}.
A rich literature investigates the regularizing effect and the fine properties that the nonlinearity of the flux function $F$ coupled with \eqref{E_mu_intro} induces 
on weak solutions. 

The one space dimensional case is special and this regularizing effect is now quite well understood. Starting from the celebrated one sided
Lipschitz estimate in \cite{Oleinik_translation} for uniformly convex fluxes several regularity results have been obtained, even for more general nonlinear fluxes \cite{ADL_note, Dafermos_inflection, Cheng_speed_BV, m_regularity}. 
The arguments essentially rely on the structure of characteristics: although the classical method of characteristics is not available for nonsmooth
solutions, some rigidity is preserved. In one space dimension two trajectories with different speed typically intersect and the interaction of the
characteristics provides the regularization. Of course this is not the case for several space dimensions.

A widely used tool to study entropy solutions to \eqref{E_CL_intro} is the kinetic formulation introduced in \cite{LPT_kinetic}, where also 
the first regularity results in terms of fractional Sobolev spaces have been obtained by means of velocity averaging lemmas.  
For further developments see \cite{TT_kinetic, GP_optimal}. Being the sign of $\mu_\eta$ not relevant
in the kinetic formulation, most of the available results in this direction hold in the more general setting where \eqref{E_mu_intro}
is replaced by
\begin{equation}\label{E_fin_intro}
\mu_\eta:=\partial_t\eta(u)+\div_x Q(u) \in \mathcal M ((0,T)\times \R^d),
\end{equation}
i.e. the entropy production measure $\mu_\eta$ is required to be locally finite but without constraints on its sign.
We refer to these solutions as \emph{weak solutions with finite entropy production}.
The first example where the sign of the entropy production is used to improve the available regularity results in the kinetic framework 
is \cite{GL_kinetic_entropic}.
In \cite{DLW_structure} (see also \cite{COW_bologna}) the authors proved that under mild nonlinearity assumptions on $f$, bounded weak solutions 
with finite entropy production share several fine properties with $BV$ functions:
more in details they proved that there exists a rectifiable set $J$ of dimension $d$ such that
\begin{enumerate}
\item $u$ has vanishing mean oscillation at every $(t,x)\notin J$;
\item $u$ has left and right traces on $J$;
\item $\mu_\eta\llcorner J= ((\eta(u^+),q(u^+)) - (\eta(u^-),q(u^-)))\cdot \mathbf n \llcorner J$, where $u^\pm$ denotes the traces on $J$ and $\mathbf n$ denotes the normal to $J$.
\end{enumerate}
For $\BV$ solutions (1) and (3) can be improved to
\begin{enumerate}
\item[(1')] every $(t,x)\notin J$ is a Lebesgue point;
\item[(3')] $\mu_\eta= ((\eta(u^+),q(u^+)) - (\eta(u^-),q(u^-)))\cdot \mathbf n \llcorner J$.
\end{enumerate}
In \cite{Silvestre_CL} the author considered the case of entropy solutions to \eqref{E_CL_intro} with a power-type nonlinearity assumption on $f$ 
(see Assumption \ref{A_f}): in this setting he proved that every point $(t,x)\notin J$ is actually a continuity point, providing therefore a 
positive answer about (1'). Moreover he showed that $\mu=\mu\llcorner \bar J$, where $\bar J$ denotes the topological closure of $J$, 
partially answering about (3').

It is also worth to mention that both questions have affirmative answer for entropy solutions in one space dimension. 
Property (1') is valid under the milder nonlinearity assumption that $\{v:f''(v)\ne 0\}$ is dense in $\R$: see \cite{BM_structure}, where it is also
proved that Property (3') holds for general smooth fluxes, see also \cite{DLR_dissipation} for an earlier proof in the case of fluxes with 
finitely many nondegenerate inflection points.
Moreover in \cite{Dafermos_continuous} it is proved that $\mu_\eta$ vanishes for continuous weak solutions, without a priori requiring that they are entropic.

Entropy solutions are of course the most relevant in the theory of scalar conservation laws, nevertheless weak solutions that are not entropic
arise naturally together with \eqref{E_fin_intro} in certain situations: in \cite{bertini_al, Mariani_large} they arise in the study of large deviations for stochastic  conservation laws.
We refer to \cite{LO_Burgers} and the reference therein for more motivations. 

Property (1') has been addressed for the first time out of the entropic setting in \cite{LO_Burgers} for the Burgers equation: 
\begin{equation}\label{E_Burgers_intro}
\partial_t u + \partial_x \left(\frac{u^2}{2}\right) =0.
\end{equation}
The authors proved that the set of non Lebesgue points of any weak solution to \eqref{E_Burgers_intro} with finite entropy production has Hausdorff dimension at most 1. It is also remarkable that they only need \eqref{E_fin_intro} to be satisfied for the entropy $\bar \eta(u)=u^2/2$.
On the other hand their argument relies on the link between Burgers equation and Hamilton-Jacobi equation, therefore it seems limited to the
one dimensional case.

In this work we obtain the analogous result in the general setting, thus providing a partial answer about (1') in the affermative.

\begin{theorem}\label{T_main_intro}
Let $u$ be a bounded weak solution to \eqref{E_CL_intro} with finite entropy production. Under a power-like nonlinearity assumption on
the flux $f$,
the set of non Lebesgue points of $u$ has Hausdorff dimension at most $d$.
\end{theorem} 
We remark here that this theorem as well as the result in \cite{DLW_structure} is proved in the more general setting of quasi-solutions (see Section \ref{S_structure}).

The proof of Theorem \ref{T_main_intro} is based on a new estimate which relates the solution to the free motion in the kinetic formulation. 
As another byproduct of this estimate we get the following theorem, which extends the case of bounded entropy solutions studied in \cite{BBM_multid}.

\begin{theorem}\label{T_lagrangian_intro}
Let $u$ be a bounded nonnegative weak solution with finite entropy production to \eqref{E_CL_intro}, with $u_0\in L^1(\R^d)$. Then $u$ admits 
a \emph{Lagrangian representation}, i.e. there exists a nonnegative bounded measure $\omega$ on
\begin{equation*}
\Gamma:=\{\gamma=(\gamma^1,\gamma^2) \in \BV([0,T),\R^d\times [0,+\infty)): \gamma^1 \text{ is Lipschitz}\}
\end{equation*}
which satisfies the following conditions:
\begin{enumerate}
\item for every $t \in [0,T)$ it holds
\begin{equation*}
(e_t)_\sharp \omega = \mathscr L^{d+1}\llcorner \{(x,v)\in\R^d\times [0,+\infty): v<u(t,x)\},
\end{equation*}
where $e_t(\gamma)=\lim_{s\to t^+}\gamma(s)$ is the evaluation map;
\item the measure $\omega$ is concentrated on characteristics, i.e. on curves $\gamma\in \Gamma$ such that
\begin{equation*}
\dot\gamma^1(t)=f'(\gamma^2(t)) \quad \mbox{for a.e. }t\in (0,T);
\end{equation*}
\item the following integral estimate holds:
\begin{equation*}
\int_\Gamma \TV_{[0,T)} \gamma^2 d\omega <\infty.
\end{equation*}
\end{enumerate}
\end{theorem}
The Lagrangian representation is of course strictly related to the kinetic formulation and it is also reminiscent of the superposition principle
for nonnegative measure valued solutions to the continuity equation \cite{Ambrosio_transport}.
By means of this notion we provide a representation of the entropy production measure $\mu_\eta$.
We notice that the existence of a Lagrangian representation, even if in a different form, has been a crucial ingredient to prove
the optimal regularizing effect \cite{m_regularity} and the structure of entropy solutions \cite{BM_structure} in the case of a single space dimension.
Moreover it is the crucial ingredient to prove in \cite{BBM_multid} that $\mu_\eta$ vanishes for bounded and continuous entropy solutions
(see also \cite{Silvestre_CL} for a similar proof of this result).

{\bf{Structure of the paper.}} In Section \ref{S_prel} we introduce the notion of quasi solutions and the corresponding kinetic formulation.
We moreover recall a few results from the theory of $L^1$ optimal transport that will be relevant in the construction of the Lagrangian
representation.
The short Section \ref{S_estimate} is devoted to the proof of the kinetic estimate. As a first consequence we deduce Theorem \ref{T_main_intro}
in Section \ref{S_structure}. Theorem \ref{T_lagrangian_intro} is proved instead in Section \ref{S_lagrangian}.

\section{Preliminaries and setting}\label{S_prel}
In this section we introduce the setting of quasi-solutions following \cite{DLW_structure} and we provide the notion of kinetic formulation.
Moreover we recall a few facts from $L^1$-Optimal Transport that will be useful in the second part of the work.
\subsection{Quasi-solutions}
We consider flux functions $f\in C^2(\R,\R^d)$.
\begin{definition}
Let $\E_+$ denote the set of all $q\in C(\R,\R^d)$ for which there exists an $\eta$ with 
\begin{equation*}
q'(v)=\eta'(v)f'(v) \qquad \mbox{and}\qquad \eta''(v)\ge 0 \quad \mbox{in} \quad \D'_v.
\end{equation*}
We call a measurable function $u:\R^n\to (0,1)$ a \emph{quasi-solution} if
\begin{equation}\label{E_def_quasi}
\mu_q:=-\div_xq(u) \,\in \, \M(\R^n) \quad \mbox{for all }q\in \E_+.
\end{equation}
\end{definition}

\begin{remark}
We observe that we can recover \eqref{E_fin_intro} from \eqref{E_def_quasi} considering $f=(\mathbf{I},F)$. Conversely quasi-solutions
can be interpreted as time-independent functions satisfying \eqref{E_fin_intro}.
Notice moreover that \eqref{E_def_quasi} does not imply $\div_x f(u)=0$. In particular this setting is more general than \eqref{E_CL_intro}, \eqref{E_fin_intro} and it allows to consider suitable sources.
Finally observe that consider quasi-solutions taking values in $(0,1)$ is not restrictive, up to  translations and rescaling of the flux $f$.
\end{remark}

In the following Proposition we introduce the kinetic formulation for quasi-solutions.

\begin{proposition}\label{P_kin}
Let $u$ be a quasi-solution and let $\chi:\R^n\times [0,+\infty)\to \{0,1\}$ be
\begin{equation*}
\chi(x,v):=
\begin{cases}
1 & \mbox{if }0<v\le u(x), \\
0 & \mbox{otherwise}.
\end{cases}
\end{equation*}
Then there exists a locally finite Radon measure $\mu\in \M(\R_v\times \R^d_x)$ such that
\begin{equation*}
f'(v)\cdot \nabla_x\chi(v,x)=\partial_v\mu \quad \mbox{in }\D'_{v,x}.
\end{equation*}
\end{proposition}
In the following we will denote by $\nu$ the $x$- marginal of the total variation $|\mu|$ of $\mu$:
\begin{equation*}
\nu(A):=|\mu|(\R\times A) \quad \mbox{for every Borel set }A\subset \R^n.
\end{equation*}

\subsection{Duality for $L^1$ optimal transport}
\begin{definition}
Let $(X,d)$ be a Polish metric space and let $\mu_1,\mu_2$ be two probability measures on $X$. The Wasserstein distance of order 1 between
$\mu_1$ and $\mu_2$ is defined by
\begin{equation}\label{E_W1}
W_1(\mu_1,\mu_2):=\inf_{\pi\in \Pi(\mu_1,\mu_2)}\int_Xd(x,y)d\pi(x,y),
\end{equation}
where $\Pi(\mu_1,\mu_2)$ is the set of transport plans from $\mu_1$ to $\mu_2$, i.e.
\begin{equation*}
\Pi(\mu_1,\mu_2):=\{\omega\in \mathcal P(X^2): {\pi_1}_\sharp \omega = \mu_1, {\pi_2}_\sharp \omega = \mu_2 \},
\end{equation*}
denoting by $\pi_1,\pi_2:X^2 \to X$ the two natural projections.
\end{definition}

Notice that $W_1$ can take value $+\infty$.

In order to prove the existence of a Lagrangian representation for weak solutions with finite entropy production the following duality formula will be useful (see for example \cite{V_oldnew}).
\begin{proposition}\label{P_duality}
For any $\mu_1,\mu_2 \in \mathcal P(X)$, it hods
\begin{equation*}
W_1(\mu_1,\mu_2)= \sup_{ \phi \in L^1(\mu_1), \|\phi\|_\Lip \le 1} \left( \int_X\phi d\mu_1-\int_X \phi d\mu_2 \right).
\end{equation*}
\end{proposition}
 
The next theorem from \cite{BD_L1map} provides the existence of an $L^1$-optimal map with respect to quite general distances on $\R^N$.
\begin{theorem}\label{T_opt_map}
Let $X=\R^N$ with $N\in\N$ be the euclidean space equipped with the distance induced by a convex norm $|\cdot|_{D*}$. Let $\mu_1,\mu_2\in \mathcal P(\R^N)$ be such that 
$\mu_1\ll \mathscr L^N$ and the infimum in \eqref{E_W1} is finite. Then there exists an optimal plan $\pi$ in \eqref{E_W1} induced by a map, i.e. there exists a measurable map $T:\R^N\to \R^N$
such that $T_\sharp \mu_1=\mu_2$ and
\begin{equation*}
W_1(\mu_1,\mu_2) = \int_X|T(x)-x|_{D*}d\mu_1(x).
\end{equation*}
\end{theorem}
We remark that the result above would be much easier requiring only the existence of an optimal plan instead of an optimal map.
This would be sufficient for our goal but this result allows for a more transparent construction in Section \ref{S_lagrangian}.

\section{A weak estimate}\label{S_estimate}
Relying on the kinetic formulation we prove in this short section the main estimate of this work.
For every $u:\R^d\to [0,+\infty)$ we denote its subgraph by
\begin{equation*}
E_u:= \{(x,v)\in \R^d\times [ 0,+\infty): v\le u(x)\},
\end{equation*}
Then we consider the \emph{free-transport} operator introduced in \cite{Brenier_TC} to approximate entropy solutions:
\begin{equation*}
\FT (E,s):=\{(x,v)\in \R^d\times [0,+\infty): (x-f'(v)s,v)\in E\}.
\end{equation*}
Moreover let $\chi_E:\R^n\times [0, + \infty)\to \{0,1\}$ be the characteristic function of the set $E$:
\begin{equation*}
\chi_E(x,v):=\begin{cases}
1 & \mbox{if }(x,v)\in E, \\
0 & \mbox{otherwise}.
\end{cases}
\end{equation*}

For any $R>0$ and $\bar x\in \R^d$ denote by $B_R(\bar x)$ the ball of radius $R>0$ and center $\bar x$. Moreover we set $\pi_x:\R^{d}\times [0,+\infty)\to \R^d$ the first projection.

\begin{theorem}\label{T_Wass1}
Let $u$ be a quasi-solution and $\bar s>0$.
Let $\phi \in C^1_c(\R^d\times (0,+\infty))$ be such that $\pi_x(\supp \phi)\subset B_R(\bar x)$. Then 
\begin{equation}\label{E_main}
\int_{\R^d\times [0,+\infty)} \phi(x,v)(\chi_{E_u}-\chi_{\FT(E_u,\bar s)})dxdv \le \left(\bar s \|\partial_v\phi\|_{L^\infty} + \frac{\bar s^2}{2}\|f''\|_{L^\infty}\|\nabla_x \phi\|_{L^\infty}\right)\nu(B_{R+ \|f'\|_\infty\bar s}(\bar x)).
\end{equation}
\end{theorem}
\begin{proof}
For every $s\in [0,\bar s]$ let 
\begin{equation*}
\chi^1(s,\cdot,\cdot):=\chi_{E_u} \qquad \mbox{and} \qquad \chi^2(s,\cdot,\cdot):=\chi_{\FT(E_u,s)}.
\end{equation*}
By Proposition \ref{P_kin} and the definition of the free transport operator we have
\begin{equation*}
\begin{split}
\partial_s \chi^1 + f'(v)\cdot \nabla_x\chi^1& = \partial_v \tilde \mu \qquad \mbox{in }\mathcal D', \\
\partial_s \chi^2 + f'(v)\cdot \nabla_x \chi^2 & = 0 \qquad \mbox{in }\mathcal D',
\end{split}
\end{equation*}
where $\tilde \mu= \mathscr L^1\times \mu$.
Let $\tilde \chi := \chi^1-\chi^2$ and $\psi(s,x,v):=\phi(x+f'(v)(\bar s - s),v)$. Then by a straightforward computation it follows that
\begin{equation}\label{E_kin_test}
\partial_s(\tilde\chi\psi)+f'(v)\cdot\nabla_x(\tilde \chi \psi)=\psi \partial_v\tilde\mu
\end{equation}
holds in the sense of distributions. Let $g:[0,\bar s]\to \R$ be defined by
\begin{equation*}
g(s) = \int_{\R^d\times (0,+\infty)}\tilde \chi(s) \psi(s) dxdv.
\end{equation*}
It follows from \eqref{E_kin_test} and the definition of $\tilde \mu$ that
\begin{equation*}
g'(s)=-\int_{\R^d\times [0,+\infty)}\partial_v\psi(s)d\mu
\end{equation*}
holds in the sense of distributions.
Therefore $g\in C^1([0,\bar s])$ and we have
\begin{equation*}
\begin{split}
\int_{\R^d\times [0,+\infty)} \phi (\chi_{E_u}-\chi_{\FT(E_u,\bar s)})dxdv = &~ g(\bar s)-g(0) \\
= &~ \int_0^{\bar s}g'(s)ds \\
=&~ -\int_0^{\bar s}\int_{\R^d\times [0,+\infty)} \partial_v\psi(s)d\mu ds \\
=&~ - \int_0^{\bar s}\int_{\R^d\times [0,+\infty)} \left(\partial_v\phi + f''(v)\cdot \nabla_x \phi \cdot (\bar s -s)\right)d\mu ds \\
\le &~ \left(\bar s \|\partial_v\phi\|_{L^\infty} + \frac{\bar s^2}{2}\|f''\|_{L^\infty}\|\nabla_x \phi\|_{L^\infty}\right)\nu(B_{R+ \|f'\|_\infty\bar s}(\bar x)),
\end{split}
\end{equation*}
which is our goal.
\end{proof}

\section{Structure of quasi-solutions}\label{S_structure}
In this section we assume the following quantitative nonlinearity estimate on the flux function $f$.
\begin{assum}\label{A_f}
There exists $\alpha\in (0,1]$ and $C>0$ so that for every $\xi\in \R^d$ with $|\xi|=1$, and any $\delta>0$, we have
\begin{equation*}
\mathscr L^1\big(\{v\in [0,1]:|f'(v)\cdot \xi| < \delta\}\big)\le C\delta^\alpha.
\end{equation*}
\end{assum}
We recall that this is the assumption that provides a fractional Sobolev regularity of the entropy solutions in  \cite{LPT_kinetic}  and
it is used in \cite{Silvestre_CL} to prove Property (1') of the introduction in the entropic setting.

As a corollary we get the following lemma.
\begin{lemma}\label{L_nonlinear}
Let $f$ be such that Assumption \ref{A_f} holds. Then there exists $\tilde c>0$ depending on $d,\|f'\|_{L^\infty}$ and $C$ from Assumption \ref{A_f} such that for every $\bar v, \bar h>0$ for which $\bar v+\bar h\le 1$ there exist
$v_1<\ldots< v_d \in [\bar v, \bar v + \bar h]$ enjoying the following property: for every $a \in \R^d$ there exists $a_1,\ldots, a_d \in \R$ so that 
\begin{equation}\label{E_claim_lemma}
a= \sum_{i=1}^da_if'(v_i), \qquad |a_i|\le \tilde c \frac{|a|}{\bar h^{d/\alpha}}, \qquad \mbox{and} \qquad |v_{i+1}-v_i|\ge \frac{\bar h}{2d}.
\end{equation}
\end{lemma}
\begin{proof}
For $j=1,\ldots,2d$ let
\begin{equation*}
I_j:=\left[\bar v +(j-1) \frac{\bar h}{2d}, \bar v + j \frac{\bar h}{2d}\right].
\end{equation*}
Let $\delta>0$ be such that $(1\vee C)\delta^\alpha=\frac{\bar h}{4d}<\frac{\bar h}{2d}$, where $C$ is the constant in Assumption \ref{A_f}. 
In particular $\delta \in (0,1)$. Then we have that there exists $v_1\in I_1$ such that
$|f'(v_1)|\ge \delta$. We denote by $\xi_1=\frac{f'(v_1)}{|f'(v_1)|}$.
By Assumption \ref{A_f} we can define inductively for $i=2,\ldots, d$ a vector $\xi_i\in S^{d-1}$ such that $\xi_i \perp \Span_{j\le i-1} \{f'(v_j)\}$ and a value $v_i\in I_{2i-1}$ such that $|f'(v_i)\cdot \xi_i|\ge \delta$. 
With this choice the requirement $|v_{i+1}-v_i|\ge \frac{\bar h}{2d}$ is satisfied and
clearly $\Span_{i\le d}\{f'(v_i)\}=\R^d$. In paticular for every $a\in \R^d$ there exists $a_1,\ldots,a_d$ such that 
\begin{equation*}
a = \sum_{i=1}^d a_if'(v_i).
\end{equation*}
Now we estimate the size of the coefficients $a_i$: inductively we prove that there exists a constant $c_i$ depending on $i$ and $\|f'\|_\infty$ such that for every $i=1,\ldots,d$ it holds $|a_{d+1-i}|\le c_i \frac{|a|}{\delta^i}$.
By the choice of $\xi_i$ we have that 
\begin{equation*}
a\cdot \xi_{d+1-i}= \sum_{j=1}^i a_{d+1-j}f'(v_{d+1-j})\cdot \xi_{d+1-i} 
\end{equation*}
therefore
\begin{equation}\label{E_est_coeff}
|a_{d+1-i}|\le \frac{|a\cdot \xi_{d+1-i}|}{|f'(v_{d+1-i})\cdot \xi_{d+1-i}|}+ \sum_{j=1}^{i-1}|a_{d+1-j}|\frac{|f'(v_{d+1-j})\cdot \xi_{d+1-i}|}{|f'(v_{d+1-i})\cdot \xi_{d+1-i}|}.
\end{equation}
For $i=1$ the estimate \eqref{E_est_coeff} says
\begin{equation*}
|a_d|\le \frac{|a\cdot \xi_{d}|}{|f'(v_{d})\cdot \xi_{d}|}\le \frac{|a|}{\delta},
\end{equation*}
so that the claim is satisfied with $c_1=1$. For $i=2,\ldots,d$ we get
\begin{equation*}
|a_{d+1-i}|\le \frac{|a|}{\delta}+ \sum_{j=1}^{i-1}|a_{d+1-j}|\frac{\|f'\|_{L^\infty}}{\delta}\le \frac{|a|}{\delta}+ \sum_{j=1}^{i-1}\frac{c_j|a|}{\delta^j}\cdot\frac{\|f'\|_{L^\infty}}{\delta},
\end{equation*}
therefore the claim is satisfied with $c_i=1+\|f'\|_{L^\infty}\sum_{j=1}^{i-1}c_j$. Letting $\bar c= \max_{i}c_i$ and exploiting the choice of $\delta$
we get that there exists $\tilde c>0$ as in the statement such that \eqref{E_claim_lemma} holds and this concludes the proof.
\end{proof}

Let us now fix some notation: for every $x\in \R^d$ and any $r>0$, denote by
\begin{equation}\label{E_def_h}
(u)_{B_r(x)}:=\frac{1}{|B_r(x)|}\int_{B_r(x)}u \qquad \mbox{and} \qquad 2\bar h_r(x)=\max_{y_1,y_2\in \bar B_{2r}(x)}\left((u)_{B_r(y_1)}-(u)_{B_r(y_2)}\right).
\end{equation}
In the following lemma we prove that $\bar h_r(x)$ can be estimated in terms of the difference between $E_u$ and an appropriate free transport of itself.

\begin{lemma}\label{L_overlapping_cyl}
Let $\tilde c>0$ be as in Lemma \ref{L_nonlinear}. Then there exists $c_1=c_1(d,f,\tilde c)>0$ such that for every $\bar x\in \R^d, r>0$ and every 
$\Delta v \in [0,c_1\bar h_{r}(\bar x)^{\frac{d}{\alpha}+1}]$ there exists $\tilde x\in \R^d$, $\tilde v \in (0,1)$ and $\tilde a\in \R$ such that
\begin{equation}\label{E_cond_1}
|\tilde a|\le \tilde c\frac{4r}{\bar h_r(\bar x)^{d/\alpha}}, \qquad |\tilde x - \bar x|\le 2r+d\|f'\|_\infty\tilde c \frac{4r}{\bar h_r(\bar x)^{d/\alpha}}
\end{equation}
and
\begin{equation*}
\mathcal L^{d+1}(\mathcal C \cap E_u) -  \mathcal L^{d+1}(\FT(E_u,\tilde a)\cap \mathcal C) \ge \bar h_r(\bar x)|B_r|\frac{\Delta v}{2d},
\end{equation*}
where
\begin{equation}\label{E_def_C}
\mathcal C:=B_r(\tilde x) \times [\tilde v, \tilde v + \Delta v].
\end{equation}
\end{lemma}
\begin{proof}
For any $x\in \R^d$, $v\in [0,1]$ and $r>0$ let us set 
\begin{equation*}
m_r(x,v):= \mathscr L^d \big(\{y\in B_r(x):u(y)>v\}\big)
\end{equation*}
and notice that by a simple application of Fubini theorem we have
\begin{equation*}
(u)_{B_r(x)}= \frac{1}{|B_r(x)|}\int_0^1 m_r(x,v)dv.
\end{equation*}
We fix $\bar x \in \R^d$ and $r>0$ and in the remaining part of the proof we simply denote $\bar h_r(\bar x)$ by $\bar h$ and $m_r$ by $m$.
Let $y_1,y_2\in B_{2r}(\bar x)$ be realizing the maximum in \eqref{E_def_h}. 
Therefore
\begin{equation*}
\begin{split}
2\bar h |B_r| &=~ \int_0^1 \left(m(y_1,v)-m(y_2,v)\right)dv \\
=&~ \int_0^{1-\bar h} \left(m(y_1,v+\bar h)-m(y_2,v)\right)dv + \int_0^{\bar h}m(y_1,v)dv - \int_{1-\bar h}^1m(y_2,v)dv \\
\le &~ \int_1^{1-\bar h}\left(m(y_1,v+\bar h)-m(y_2,v)\right)dv  + \bar h |B_r|.
\end{split}
\end{equation*}
In particular there exists $\bar v\in [0,1-\bar h]$ such that 
\begin{equation}\label{E_diff}
m(y_1,\bar v+ \bar h)-m(y_2,\bar v) \ge \bar h |B_r|.
\end{equation}
Now we are in position to apply Lemma \ref{L_nonlinear} with the choice of $\bar v$ and $\bar h$ as above and with $a=y_1-y_2$.
Let $v_1,\ldots,v_d$ be from Lemma \ref{L_nonlinear} and set $x_0:=y_2$ and inductively $x_i:=x_{i-1}+a_if'(v_i)$ for $i=1,\ldots,d$.
We moreover set $v_0=\bar v$ and $v_{d+1}=\bar v + \bar h$. By construction $y_1=x_d$ therefore by \eqref{E_diff}
we have that $m(x_{d}, v_{d+1})-m(x_0,v_0)\ge \bar h |B_r|$. Since $v_0\le v_1$ it holds $m(x_0,v_1)\le m(x_0,v_0)$ therefore 
\begin{equation*}
\bar h |B_r|\le  m(x_{d}, v_{d+1})-m(x_0,v_1) = \sum_{i=1}^d m(x_i,v_{i+1})-m(x_{i-1},v_i).
\end{equation*}
In particular there exists $l\in \{1,\ldots, d\}$ such that 
\begin{equation}\label{E_bad_l}
m(x_l,v_{l+1})-m(x_{l-1},v_l)\ge \frac{\bar h |B_r|}{d}.
\end{equation}
We set
\begin{equation*}
\tilde a := a_l, \qquad \tilde v:=v_l \qquad \mbox{and} \qquad \tilde x:=x_l.
\end{equation*}
Notice that the two conditions in \eqref{E_cond_1} are satisfied by Lemma \ref{L_nonlinear}.

Given $x\in \R^d$, $v\in [0,1]$ and $s\in \R$ set
\begin{equation*}
m'(x,v,s):= m(x-f'(v)s,v) = \mathscr L^d\big(\{y\in B_r(x):(y,v)\in \FT(E_u,s)\}\big).
\end{equation*}
We now prove the following claim.

\textbf{Claim.}
There exists a constant $c_1=c_1(d,f,\tilde c)>0$ such that if $0\le \Delta v \le c_1 \bar h^{\frac{d}{\alpha}+1}$, then
for every $v_1\in [v_l,v_l+\Delta v]$ and any $v_2\in [v_{l+1}-\Delta v,v_{l+1}]$ it holds
\begin{equation*}
m(x_l,v_2)-m'(x_l,v_1,a_l) \ge \frac{\bar h|B_r|}{2d}.
\end{equation*}
\noindent
\emph{Proof of the claim.} By the monotonicity of $m$ with respect to $v$ and \eqref{E_bad_l} it holds 
\begin{equation*}
\begin{split}
m(x_l,v_2)-m'(x_l,v_1,a_l)&= m(x_l,v_2) - m(x_l,v_{l+1}) + m(x_l,v_{l+1}) - m(x_{l-1},v_l)+ m(x_{l-1},v_l)- m'(x_l,v_1,a_l) \\
\ge &~ \frac{\bar h|B_r|}{d}+ m(x_l-a_lf'(v_l),v_l)- m (x_l-a_lf'(v_1),v_1) \\
\ge &~ \frac{\bar h|B_r|}{d}+ m(x_l-a_lf'(v_l),v_l)- m (x_l-a_lf'(v_1),v_l).
\end{split}
\end{equation*}
Therefore it is sufficient to prove that there exists $c_1$ as in the statement such that if $0\le \Delta v\le c_1 \bar h^{\frac{d}{\alpha}+1}$, then
\begin{equation}\label{E_est_deltav}
m (x_l-a_lf'(v_1),v_l) - m(x_l-a_lf'(v_l),v_l) \le  \frac{\bar h|B_r|}{2d}.
\end{equation}
Clearly it holds
\begin{equation*}
\begin{split}
m (x_l-a_lf'(v_1),v_l) - m(x_l-a_lf'(v_l),v_l) &\le~ |B_r(x_l-a_lf'(v_1)) \Delta B_r(x_l-a_lf'(v_l))| \\
&\le~ C_d r^{d-1}|a_l|\|f''\|_{L^\infty}\Delta v.
\end{split}
\end{equation*}
Since by Lemma \ref{L_nonlinear} it holds $|a_l|\le \tilde c \frac{|a|}{\bar h^{\frac{d}{\alpha}}}$ and $|a|=|y_2-y_1|\le 4r$, we get
\begin{equation*}
m (x_l-a_lf'(v_1),v_l) - m(x_l-a_lf'(v_l),v_l) \le \frac{4C_dr^d\tilde c \|f''\|_{L^\infty}}{\bar h^{\frac{d}{\alpha}}}\Delta v,
\end{equation*}
which implies \eqref{E_est_deltav} if $0\le \Delta v\le c_1 \bar h^{\frac{d}{\alpha}+1}$ with 
$c_1=\omega_d\left(8dC_d\tilde c \|f''\|_{L^\infty}\right)^{-1}$. This concludes the proof of the claim.

Let us now consider the cylinders $\mathcal C$ defined in \eqref{E_def_C} and
\begin{equation*}
\mathcal C' := B_r(x_l)\times [v_{l+1}-\Delta v, v_{l+1}].
\end{equation*}
By the previous claim and the monotonicity of $m$ with respect to $v$ we have
\begin{equation*} 
\begin{split}
 \mathcal L^{d+1}(\mathcal C \cap E_u) -  \mathcal L^{d+1}(\FT(E_u,a_l)\cap \mathcal C)  \ge& ~ 
 \mathcal L^{d+1}(\mathcal C' \cap E_u) -  \mathcal L^{d+1}(\FT(E_u,a_l)\cap \mathcal C) \\
  = &~ \int_0^{\Delta v} m(x_l,v_{l+1}-\Delta v + v) - m'(x_l,v_l+v,a_l) dv\\
\ge &~ \frac{\bar h |B_r|}{2d}\Delta v.
\end{split}
\end{equation*}
and this concludes the proof of the lemma.
\end{proof}

In the next proposition we take advantage of the last lemma to build an appropriate test function $\phi$ in Theorem \ref{T_Wass1} and estimate
$\bar h_r(\bar x)$ in terms of the dissipation measure $\nu$.

\begin{proposition}\label{P_Kant}
There exist $C_1=C_1(d,f)>0$, $C_2=C_2(d,f,c_1,\tilde c)>0$ and $\gamma=\gamma(d,\alpha)>0$ such that for every $r>0$ there exists $r_2\in [r,C_1 r/\bar h^{d/\alpha}]$ for which
\begin{equation*}
\nu(B_{r_2})\ge C_2 \bar h^\gamma r_2^{d-1}.
\end{equation*}
\end{proposition}
\begin{proof}
Let $a,b>0$ and $\psi_{a,b}:[0,+\infty)\to [0,1]$ be a smooth function such that 
\begin{enumerate}
\item $\psi(t)=1$ for every $t\in[0,a]$,
\item $\psi(t)=0$ for every $t\ge a+b$,
\item $|\psi'(t)|\le \frac{2}{b}$ for every $t > 0$.
\end{enumerate}
Let $r', v'>0$ be two parameters that will be fixed later and let $\tilde x,\tilde v,\tilde a, \Delta v$ be as in Lemma \ref{L_overlapping_cyl}.
Consider the function $\phi:\R^d_x\times [0,+\infty)_v\to \R$ defined by
\begin{equation*}
\phi(x,v):=\psi_{r,r'}(|x-\tilde x|)\psi_{\Delta v/2, v'}\left(\left|v-\left(\tilde v+\frac{\Delta v}{2}\right)\right|\right).
\end{equation*}
Then the following estimate holds:
\begin{equation}\label{E_est_cyl1}
\begin{split}
\int_{\R^d\times [0,+\infty)} \phi &(\chi_{E_u}-\chi_{\FT(E_u,\tilde a)})dxdv = ~ \int_{\mathcal C} \phi (\chi_{E_u}-\chi_{\FT(E_u,\tilde a)})dxdv +  \int_{(\supp \phi) \setminus \mathcal C} \phi (\chi_{E_u}-\chi_{\FT(E_u,\tilde a)})dxdv \\
= &~ \mathcal L^{d+1}(\mathcal C\cap E_u) -\mathcal L^{d+1}(\mathcal C \cap \FT(E_u,\tilde a)) +  \int_{(\supp \phi) \setminus \mathcal C} \phi (\chi_{E_u}-\chi_{\FT(E_u,\tilde a)})dxdv \\
\ge &~ \frac{\bar h |B_r|}{2d}\Delta v  -\mathcal L^{d+1}(((\supp \phi) \setminus \mathcal C) \cap (\R^d \times [0,1])).
\end{split}
\end{equation}
By an elementary geometric consideration and assuming $r'\le r$ we get
\begin{equation}\label{E_est_cyl2}
\begin{split}
\mathcal L^{d+1}(((\supp \phi) \setminus \mathcal C) \cap (\R^d \times [0,1])) \le &~ \omega_d (r+r')^d((\Delta v+2v')\wedge 1)-\omega_dr^d \Delta v \\
\le &~ c_d r^{d-1}r' + 2\omega_dr^d v',
\end{split}
\end{equation}
where $c_d$ is a geometric constant depending only on the dimension $d$.
It follows from \eqref{E_est_cyl1} and \eqref{E_est_cyl2} that under the constraints
\begin{equation}\label{E_const_r'_v'}
 c_d r^{d-1}r' \le \frac{\bar h |B_r|}{8d}\Delta v \qquad \mbox{and} \qquad 2\omega_dr^d v'\le \frac{\bar h |B_r|}{8d}\Delta v
\end{equation}
it holds
\begin{equation*}
\int_{\R^d\times [0,+\infty)} \phi (\chi_{E_u}-\chi_{\FT(E_u,a_l)})dxdv  \ge \frac{\bar h |B_r|}{4d}\Delta v.
\end{equation*}
By \eqref{E_main} we deduce
\begin{equation}\label{E_together}
\frac{\bar h |B_r|}{4d}\Delta v \le \left( \frac{2|\tilde a|}{v'} + \frac{\tilde a^2}{r'}\|f''\|_{L^\infty}\right)\nu(B_{r+r'+\|f'\|_{L^\infty}|a_l|}).
\end{equation}
By \eqref{E_cond_1} and assuming $r'\le r$ we have that
\begin{equation}\label{E_R_le_r}
R:=r+r'+ \|f'\|_{L^\infty}|a_l| \le 2r+ 4\tilde c\|f'\|_{L^\infty}\frac{r}{\bar h^{d/\alpha}}\le (2+4\tilde c \|f'\|_{L^\infty})\frac{r}{\bar h^{d/\alpha}}.
\end{equation}
By Lemma \ref{L_overlapping_cyl} we can choose $\Delta v= c_1\bar h^{\frac{d}{\alpha}+1}$ in  \eqref{E_together}, and again by
\eqref{E_cond_1} we have that for every $r'\le r$ and $v'$ satisfying \eqref{E_const_r'_v'} it holds
\begin{equation*}
\frac{\bar h |B_r|}{4d}  c_1\bar h^{\frac{d}{\alpha}+1}\le \left( \frac{8\tilde c r}{\bar h^{d/\alpha}v'} +
\frac{16\tilde c^2 \|f''\|_{L^\infty}r^2}{\bar h^{2d/\alpha}r'}\right) \nu(B_R).
\end{equation*}
In particular at least one of the two addends in the right hand side must be bigger than half the left hand side, i.e. at least one of the following inequalities holds:
\begin{equation}\label{E_alternative}
\frac{\omega_dr^{d-1}c_1v'\bar h^{\frac{2d}{\alpha}+2}}{32d\tilde c}\le \nu (B_R), \qquad 
\frac{\bar h^{\frac{3d}{\alpha}+2}\omega_dr^{d-2}c_1r'}{64d\tilde c^2\|f''\|_{L^\infty}}\le \nu (B_R).
\end{equation}
We choose now $v'$ and $r'$ as follows:
\begin{equation}\label{E_def_r'_v'}
 v':=  \, \frac{\bar h \Delta v}{16 d} = \frac{c_1\bar h^{\frac{d}{\alpha}+2}}{16d}
 \qquad \mbox{and} \qquad 
 r':= \, r \wedge \left( \frac{\bar h |B_r| \Delta v}{8dc_dr^{d-1}} \right) =
 r \wedge \left( \frac{c_1\bar h^{\frac{d}{\alpha}+2}\omega_dr}{8dc_d}\right).
\end{equation}
In particular the constraints $r'\le r$ and \eqref{E_const_r'_v'} are satisfied.
The first inequality in \eqref{E_alternative} reads
\begin{equation}\label{E_ineq_1}
\frac{\omega_dc_1^2}{256d^2\tilde c}\bar h^{\frac{3d}{\alpha}+4}r^{d-1}\le \nu(B_R).
\end{equation}
The second inequality in \eqref{E_alternative} implies
\begin{equation}\label{E_ineq_23}
\frac{\omega_dc_1}{64d\tilde c^2\|f''\|_{L^\infty}}\bar h^{\frac{3d}{\alpha}+2}r^{d-1}\le \nu (B_R) \qquad \mbox{or} \qquad 
\frac{\omega_d^2c_1^2}{512d^2c_d\tilde c^2\|f''\|_{L^\infty}}\bar h^{\frac{4d}{\alpha}+4}r^{d-1} \le \nu (B_R),
\end{equation}
depending on in which terms the minimum in \eqref{E_def_r'_v'} is attained. So we have that at least one of the three inequalities from 
\eqref{E_ineq_1} and \eqref{E_ineq_23} holds true. Therefore there exists $\tilde c_1>0$ depending on $d,c_1,\tilde c,f$ such that
\begin{equation}\label{E_r_R}
\tilde c_1 \bar h^{\frac{4d}{\alpha}+4}r^{d-1}\le \nu (B_R).
\end{equation}
The last step is to replace $r$ with $R$ in \eqref{E_r_R}. From \eqref{E_R_le_r} we have that
\begin{equation*}
r^{d-1}\ge R^{d-1} \bar h^{\frac{d(d-1)}{\alpha}}\frac{1}{(2+4\tilde c\|f'\|_{L^\infty})^{d-1}}.
\end{equation*}
Therefore we get from \eqref{E_r_R} that the statement holds true with
\begin{equation*}\label{E_def_gamma}
\gamma = \frac{d(d+3)}{\alpha}+4 \qquad \mbox{and} \qquad C_2= \frac{\tilde c_1}{(2+4\tilde c\|f'\|_{L^\infty})^{d-1}}
\end{equation*}
and this concludes the proof.
\end{proof}

In the following corollary we deduce a power decay of $\bar h_r(\bar x)$ from the power decay of $\nu(B_r(\bar x))r^{d-1}$.

\begin{corollary}\label{C_oscillation}
Let $\bar x\in \R^d$, $R,C>0$ and $\beta>0$ be such that $\nu(B_r(\bar x))\le C(r^{d-1+\beta})$ for $r\in (0,R)$.
Then there exists $\gamma'>0$ such that 
\begin{equation*}\label{E_diss_pol}
\bar h_r(\bar x)=O(r^{\gamma'}) \qquad \mbox{as }r\to 0.
\end{equation*}
\end{corollary}
\begin{proof}
Let
\begin{equation*}
\gamma'=\frac{\beta}{\gamma}\left( 1+\frac{d\beta}{\alpha\gamma}\right)^{-1}.
\end{equation*}
Assume that $\bar h_r (\bar x)\ge r^{\gamma'}$. Then by Proposition \ref{P_Kant} there exists $r_2\in [r,C_1\frac{r}{\bar h_r^{d/\alpha}(\bar x)}]$ such that 
\begin{equation*}
\bar h_r(\bar x)\le \left(\frac{\nu(B_{r_2}(\bar x))}{C_2r_2^{d-1}}\right)^{\frac{1}{\gamma}} \le \left(\frac{C}{C_2}\right)^{\frac{1}{\gamma}} r_2^{\beta/\gamma}\le \left(\frac{CC_1^\beta}{C_2}\right)^{\frac{1}{\gamma}}\frac{r^{\beta/\gamma}}{\bar h_r^{d\beta/\alpha\gamma}(\bar x)}
\le  \left(\frac{CC_1^\beta}{C_2}\right)^{\frac{1}{\gamma}}\frac{r^{\beta/\gamma}}{r^{\gamma'd\beta/\alpha\gamma}}= \left(\frac{CC_1^\beta}{C_2}\right)^{\frac{1}{\gamma}}r^{\gamma'}.
\end{equation*}
This proves that for sufficiently small $r$ it holds
\begin{equation*}
\bar h_r(\bar x)\le \left(1\vee \left(\frac{CC_1^\beta}{C_2}\right)^{\frac{1}{\gamma}}\right) r^{\gamma'}. \qedhere
\end{equation*}
\end{proof}

\begin{lemma}\label{L_polynomial}
Let $\bar x\in \R^d$ be a point of vanishing mean oscillation of $u$ such that $\exists \gamma'>0$ for which
\begin{equation*}
\bar h_r(\bar x)=O(r^{\gamma'}) \qquad \mbox{as }r\to 0.
\end{equation*}
Then $\bar x$ is a Lebesgue point of $u$.
\end{lemma}
\begin{proof}
It is sufficient to prove that there exists $\lim_{r\to 0}(u)_{B_r(\bar x)}$.
In the following of this proof we will not specify the center $\bar x$ and we will write $u_r$ for $(u)_{B_r(\bar x)}$.

We assume the following elementary fact which follows from Fubini theorem: for every $r>0$ there exists $r' \in [2r,3r]$ such that
\begin{equation*}
|u_{r'}-u_r| \le \bar h_r.
\end{equation*}
For $r>0$ denote by
\begin{equation*}
\e(r):=\fint_{B_{r}}|u-u_{r}|.
\end{equation*}
We notice that if $r \in [r'/3,r']$, then $|u_r-u_{r'}|\le 3^d\e(r')$.
In fact
\begin{equation*}
\begin{split}
\e(r') = &~ \fint_{B_{r'}}|u-u_{r'}| \\
\ge &~ \frac{|B_r|}{|B_{r'}|} \fint_{B_r} |u-u_{r'}| \\
\ge &~ \frac{|B_r|}{|B_{r'}|} |u_r-u_{r'}| \\
= &~ 3^{-d} |u_r-u_{r'}|.
\end{split}
\end{equation*}
We prove that $(u)_{B_r}$ is a Cauchy sequence as $r\to 0$.
Let $0<r<R$. If $r>R/3$ then $|u_r-u_R|\le 3^d\e(R)$, otherwise let $r_1\in [2r,3r]$ be such that $|u_r-u_{r_1}|\le \bar h_r$. Iterating this argument we have that there exist $n\in\N$ and
$r_1,\ldots, r_n$ such that $|u_{r_i}-u_{r+1}|\le \bar h_{r_i}$, $r_i \in [2^ir,3^ir]$ and $u_{r_n}\in [R/3,R]$.
So we have
\begin{equation*}
\begin{split}
|u_r-u_R|&\le |u_r-u_{r_1}| + \sum_{i=1}^{n-1} |u_{r_i}-u_{r_{i+1}}| + |u_{r_n}-u_R| \\
& \le \bar h (r) + \sum_{i=1}^{n-1}\bar h(r_i) + \e(R) \\
& \le C\left(r^{\gamma'} + \sum_{i=1}^{n-1} r_i^{\gamma'}\right) + \e(R) \le C_{\gamma'}R^{\gamma'} + \e(R),
\end{split}
\end{equation*}
which converges to 0 as $R\to 0$ and this proves the lemma. 
\end{proof}

Let $J\subset \tilde J \subset \R^d$ defined by:
\begin{equation*}
\begin{split}
J^c & :=\big\{x\in \R^d: \nu (B_r(x))=o\left(r^{d-1}\right) \mbox{ as }r\to 0 \big\},\\
\tilde J^c&:=\big\{x\in \R^d: \nu(B_r(x))=O\left(r^{d-1 +\alpha}\right) \mbox{ for some }\alpha>0 \mbox{ as }r\to 0  \big\}. 
\end{split}
\end{equation*} 
Let moreover $R\subset \R^d$ the set of Lebesgue points of $u$. In \cite{DLW_structure} is proved that every point in $J^c$ is a
vanishing mean oscillation point of $u$. Therefore it immediately follows by Corollary \ref{C_oscillation} and Lemma \ref{L_polynomial} that 
each point in $\tilde J^c$ is a Lebesgue point of $u$ so that 
\begin{equation}\label{E_inclusions}
\tilde J^c \subset R\subset J^c.
\end{equation}

We state the main result in the following theorem. At this point
the argument of the proof is the same as in \cite{LO_Burgers}. We sketch it here for completeness.
\begin{theorem}
Let $f$ be a flux satisfying Assumption \ref{A_f} and let $u$ be a quasi-solution.
Then the set $R^c$ of non Lebesgue points of $u$ has Hausdorff dimension at most $d-1$.
\end{theorem}
\begin{proof}
In view of \eqref{E_inclusions} it is sufficient to check we check that $\tilde J$ has Hausdorff dimension at most $d-1$: 
let $\alpha,K,R>0$ and set
\begin{equation*}
E_{\alpha,K,R}:=\big\{x\in B_R \subset \R^d: \nu(B_r(x))\le Kr^{d-1+\alpha} \, \forall r\in (0,1)\big\}.
\end{equation*}
By Vitali covering theorem it follows that $\mathscr H^{d-1+\alpha}(B_R\setminus E_{\alpha,K,R})\lesssim K^{-1}\nu(B_{R+1})$.
Therefore setting $E_{\alpha,K}:=\bigcup_{R>0}E_{\alpha,K,R}$ we have that $E_{\alpha,K}^c$ has Hausdorff dimension at most $d-1+\alpha$.
Being
\begin{equation*}
\tilde J = \bigcap_{\alpha,K>0}E_{\alpha,K}^c
\end{equation*}
it has Hausdorff dimension at most $d-1$.
\end{proof}
In the last part of this section we notice with a simple example that the inclusion $R\subset J^c$ can be strict:
in particular we provide a quasi-solution to the Burgers equation \eqref{E_Burgers} on $\R^2$ for which the origin does not belong to $J$ and it is not a Lebesgue point of $u$:
\begin{equation}\label{E_Burgers}
\partial_t u + \partial_x\left(\frac{u^2}{2}\right)=0.
\end{equation}
This shows that Property (1') in the introduction is not true in general, the condition $\mathscr H^{d-1}(R^c\setminus J)=0$ would be satisfactory as well, but we cannot prove it or disprove it here.

{\bf{Example.}}
Let $u_0 \in L^\infty(\R)$ be such that 0 is a vanishing mean oscillation point of $u_0$ but not
a Lebesgue point: consider for example $u_0(x)=\sin(\log|\log|x||)$ and let $u_1:[0,+\infty)\times \R$ be the entropy solution to the Cauchy problem for \eqref{E_Burgers} with initial datum $u_0$. Moreover let $u_2:[0,+\infty)\times \R$ be the entropy solution of the Cauchy problem
\begin{equation}\label{E_minus_Burgers}
\begin{cases}
\partial_t u - \partial_x\left(\frac{u^2}{2}\right)=0, \\
u(0,\cdot)=u_0
\end{cases}
\end{equation}
and set 
\begin{equation*}
u(t,x)= \begin{cases}
u_1(t,x) & \mbox{if }t\ge 0; \\
u_2(-t,x) & \mbox{if }t<0.
\end{cases}
\end{equation*}
Being $u_1,u_2\in C([0,+\infty);L^1_\loc(\R))$ it is straightforward to check that $u$ is a quasi solution on the whole $\R^2$.
We now check that the origin does not belong to $J$ and that it is not a Lebesgue point in the two variables $(t,x)$.
Let us denote by
\begin{equation*}
(u_0)_r:=\frac{1}{2r}\int_{-r}^ru_0(x)dx, \qquad \e(u_0,r):= \frac{1}{2r}\int_{-r}^r|u_0(x)-(u_0)_r|dx.
\end{equation*}
By Kruzkov contraction estimate in $L^1(\R)$ we have that for any $r>0$ and any $t\in [0,2r]$ it holds
\begin{equation}\label{E_460}
\frac{1}{4r}\int_{-2r}^{2r}|u(t,x)-(u_0)_{4r}|dx \le \frac{1}{4r}\int_{-4r}^{4r}|u_0(x)-(u_0)_{4r}|dx = 2\e(u_0,4r),
\end{equation}
so that integrating for $t \in[0,2r]$ we get
\begin{equation}\label{E_Leb_point}
\int_0^{2r}\int_{-2r}^{2r}|u(t,x)-(u_0)_{4r}|dxdt \le 16 \e(u_0,4r)r^2.
\end{equation}
Since $(u_0)_r$ is not converging as $r\to 0$ and $\e(u_0,r)\to 0$ as $r\to 0$ this proves that the origin is not a Lebesgue point of $u$.
Moreover it follows from \eqref{E_Leb_point} by Fubini theorem that there exist $s_1\in [-2r,-r], s_2\in[r,2r]$ such that
\begin{equation}\label{E_462}
\int_0^{2r}|u(t,s_1)-(u_0)_{4r}| dt \le 16 \e(u_0,4r)r, \qquad \int_0^{2r}|u(t,s_2)-(u_0)_{4r}| dt \le 16 \e(u_0,4r)r.
\end{equation}
Computing the balance for the entropy $\bar \eta(u)=u^2/2$ on the domain $D:=(0,2r)\times (s_1,s_2)$ we get 
by \eqref{E_460} and \eqref{E_462}
\begin{equation*}
\begin{split}
|\mu_{\bar \eta}(D)| \le &~ \int_{s_1}^{s_2} |\bar \eta(u(2r,x))-\bar \eta(u_0(x))| dx + \int_0^{2r}|\bar q(u(t,s_2))-\bar q(u(t,s_1))|dt \\
\le &~  \int_{s_1}^{s_2} \left( |\bar \eta(u(2r,x))-\bar \eta((u_0)_{4r})| + |\bar \eta((u_0)_{4r})-\bar \eta(u_0(x))| \right) dx \\
& + \int_0^{2r} \left( |\bar q(u(t,s_2))-\bar q((u_0)_{4r})| + |\bar q((u_0)_{4r}) - \bar q(u(t,s_1))| \right) dt \\
\le &~ 48 \e(u_0,4r) r,
\end{split}
\end{equation*}
being $\bar \eta$ and $\bar q(u)=u^3/3$ both 1-Lipschitz functions on $[-1,1]$. The same computation holds for $t<0$ and since for entropy
solutions to Burgers equation and to \eqref{E_minus_Burgers} it holds $-\mu_{\bar \eta}=|\mu_{\bar \eta}|=\nu$ this proves that 
$\nu(B_r(0))=o(r)$, i.e. $0\notin J$.

\section{Lagrangian representation for the time dependent case}\label{S_lagrangian}
In this section we consider the Cauchy problem for the scalar conservation law:
\begin{equation}\label{E_multiD}
\begin{cases}
& u_t + \div_x F(u)=0, \\
& u(0,\cdot)=u_0,
\end{cases}
\end{equation}
with $u:[0,T)\times \R^d\to \R$ a measurable function for some $T>0$, $u_0\in L^1(\R^d)\cap L^\infty(\R^d)$ and $F\in C^2(\R,\R^d)$. 
\begin{definition}
We say that $u\in C([0,T];L^1(\R^d))$ is a \emph{weak solution with finite entropy production} if it solves \eqref{E_multiD} in the sense of
distributions and for every entropy $\eta\in C^2(\R)$ such that $\eta''(v)\ge 0$ and corresponding flux $Q:\R\to \R^d$ satisfying 
$Q'=\eta'F'$ the distribution
\begin{equation}\label{E_diss}
\mu_Q:=\eta(u)_t+\div_x Q(u) 
\end{equation}
is a finite Radon measure in $[0,T]\times \R^d$.
\end{definition}
\begin{remark}
We notice that the existence of an $L^1$ continuous representative in time of a weak solution $u$ with finite entropy production can be deduced
from \eqref{E_diss} under the assumption of genuine nonlinearity of the flux $F$ (see \cite{dafermos_book_4}).
\end{remark}

In order to keep the presentation simpler we restrict our attention to the following class of solutions.
\begin{assumption}\label{A_u}
The weak solution $u:[0,T)\times \R^d \to \R$ to \eqref{E_multiD} is bounded, nonnegative and $\|u_0\|_{L^1(\R^d)}=1$.
\end{assumption}

In this context the kinetic formulation has the following form:
\begin{proposition}\label{P_kin_t}
Let $u$ be a weak solution with finite entropy production and let $\chi:[0,T)\times\R^d\times [0,+\infty)\to \{0,1\}$ be
\begin{equation*}
\chi(t,x,v):=
\begin{cases}
1 & \mbox{if }0<v\le u(t,x), \\
0 & \mbox{otherwise}.
\end{cases}
\end{equation*}
Then there exists a finite Radon measure $\mu\in \M([0,T)\times\R^d\times \R)$ such that
\begin{equation}\label{E_kin_t}
\partial_t \chi + f'(v)\cdot \nabla_x\chi=\partial_v\mu \quad \mbox{in }\D'_{t,x,v}.
\end{equation}
\end{proposition}
We denote by $\nu$ the projection on $[0,T)\times \R^d$ of the total variation $|\mu|$ of $\mu$ and we notice that the $L^1$ continuity in time
of $u$ implies that $(\pi_t)_\sharp \nu \in \mathcal M([0,T))$ has no atoms.

In order to introduce the notion of Lagrangian representation we set some notation: we denote by
\begin{equation*}
\begin{split}
\tilde \Gamma &:= \BV([0,T);\R^d\times [0,+\infty)), \\
\Gamma& :=\left\{\gamma=(\gamma^1,\gamma^2)\in \tilde \Gamma: \gamma^1 \mbox{ is Lipschitz}
\right\}.
\end{split}
\end{equation*}
In order to fix a representative we will also assume that $\gamma$ is continuous from the right.
We will consider on $\Gamma$ and $\tilde \Gamma$ the topology $\tau$ obtained as the product of the uniform convergence on compact sets topology for $\gamma^1$ and the $L^1$
topology for $\gamma^2$.

For every $t\in [0,T)$ let $e_t:\tilde \Gamma\to \R^d\times [0,+\infty)$ be the evaluation map:
\begin{equation*}
e_t(\gamma):=\lim_{s\to t^+}\gamma(s)=\gamma(t).
\end{equation*}
\begin{definition}
Let $u$ be a weak solution to \eqref{E_multiD} with finite entropy production satisfying Assumption \ref{A_u}. We say that $\omega \in \M(\Gamma)$ is a \emph{Lagrangian representation} of $u$ if the following conditions hold:
\begin{enumerate}
\item for every $t\in [0,T)$ it holds
\begin{equation}\label{E_repr_formula}
(e_t)_\sharp \omega = \mathscr L^{d+1}\llcorner E_{u(t)};
\end{equation}
\item the measure $\omega$ is concentrated on the set of curves $\gamma\in \Gamma$ such that
\begin{equation}\label{E_characteristic}
\dot\gamma^1(t)=f'(\gamma^2(t)) \quad \mbox{for a.e. }t\in [0,T);
\end{equation}
\item 
\begin{equation}\label{E_reg}
\int_\Gamma \TV_{[0,T)} \gamma^2 d\omega(\gamma) <\infty.
\end{equation}
\end{enumerate}
\end{definition}
A few comments are in order: 
the condition \eqref{E_repr_formula} encodes the link between the measure $\omega$ and the weak solution $u$, while \eqref{E_characteristic}
says that the mass is transported with the characteristic speed.
Finally \eqref{E_reg} is just a regularity requirement and it is related to the finiteness of the entropy production. This connection will be made more explicit in the propositions \ref{P_kin_meas} and \ref{P_any_ent}.

With the same notation as in Section \ref{S_estimate} we can state in this setting the analogous of Theorem \ref{T_Wass1}, exploiting the special role of the variable $t$ and the conservation of $\|u(t)\|_{L^1(\R^d)}$.
\begin{proposition}
Let $u$ be a weak solution to \eqref{E_multiD} with finite entropy production satisfying Assumption \ref{A_u}. Let moreover $\bar s>0$ and 
$\phi \in C^1_c(\R^d\times (0,+\infty))$ be such that $\pi_x(\supp \phi) \subset B_R(\bar x)$ for some $\bar x\in \R^d$ and $R>0$. Then for every $t>0$ it holds
\begin{equation}\label{E_Wass_t}
\int_{\R^d\times [0,+\infty)}\phi(x,v)(\chi_{E_{u(t+\bar s)}}-\chi_{\FT(E_{u(t)},\bar s)})dxdv \le \left(\|\partial_v\phi\|_{L^\infty} +\bar s\|f''\|_{L^\infty}\|\nabla_x \phi\|_{L^\infty}\right)\nu((t,t+\bar s)\times B_{R+ \|f'\|_\infty\bar s}(\bar x)).
\end{equation}
\end{proposition}
We omit the proof of this proposition since it is analogous to the proof of Theorem \ref{T_Wass1}.

We are going to consider \eqref{E_Wass_t} for small $\bar s$. 
In view of our interpretation through Proposition \ref{P_duality} the additional factor $\bar s$ in the second term of the right hand side
corresponds to a different behavior of the horizontal and vertical displacements. This is why we are going to consider anisotropic distances.

Let $L>0$; we denote by
\begin{equation*}
d_L((x_1,v_1),(x_2,v_2)):=L|x_1-x_2|+|v_1-v_2|.
\end{equation*}
We set $X=\R^d\times [0,+\infty)$ and we denote by $W_1^L$ the Wasserstein distance on $\mathcal P((X,d_L))$.

\begin{corollary}\label{C_map}
Let $u$ be a weak solution to \eqref{E_multiD} with finite entropy production satisfying Assumption \ref{A_u}. Let moreover $L>0$ and $t,\bar s\ge 0$ be
such that 
\begin{equation}\label{E_cond_s}
\bar s \le \frac{1}{\|f''\|_{L^\infty}L^2}.
\end{equation}
Then there exists $T=(T^1,T^2):\R^d\times [0,+\infty) \to \R^d\times [0,+\infty)$ such that
\begin{equation}\label{E_est_T}
\begin{split}
& T_\sharp \left(\mathscr L^{d+1}\llcorner \FT(E_{u(t)},\bar s)\right) = \mathscr L^{d+1}\llcorner E_{u(t+\bar s)}, \\
& \int_{\FT(E_{u(t)},\bar s)}\left( L|T^1(x,v)-x| + |T^2(x,v)-v|\right)dxdv \le \left( 1+\frac{1}{L}\right)\nu((t,t+\bar s)\times\R^d).
\end{split}
\end{equation}
\end{corollary}
\begin{proof}
Let $\phi\in C^1(\R^d\times (0,+\infty))$ be 1-Lipschitz with respect to $d_L$. This is equivalent to require that $\|\nabla_x\phi\|_{L^\infty}\le L$ and 
$\|\partial_v\phi\|_{L^\infty}\le 1$. From \eqref{E_Wass_t} it follows that
\begin{equation*}
\sup_{\Lip_{d_L}(\phi)\le 1} \int_{\R^d\times [0,+\infty)}\phi(x,v)(\chi_{E_{u(t+\bar s)}}-\chi_{\FT(E_{u(t)},\bar s)})dxdv \le 
\left(1 + \bar s L\|f''\|_{L^\infty}\right)\nu((t,t+\bar s)\times\R^d).
\end{equation*}
By Proposition \ref{P_duality} and \eqref{E_cond_s} it follows that
\begin{equation*}
W_1^L(\FT(E_{u(t)},\bar s), E_{u(t+\bar s)})\le \left( 1+\frac{1}{L}\right)\nu((t,t+\bar s)\times\R^d).
\end{equation*}
The conclusion follows from Theorem \ref{T_opt_map}.
\end{proof}

\subsection{Approximation scheme}
In this part we build an approximate Lagrangian representation by means of the free transport operator and Corollary \ref{C_map}.

Given $T>0$ and $n\in \N$ we set $\bar s_n= 2^{-n}T$ and $L_n= (\bar s_n \|f''\|_{L^\infty})^{-1/2}$ so that $\bar s_n$ and $L_n$ satisfy
\eqref{E_cond_s}. For every $k=1,\ldots, 2^n-1$ let $T_k$ be an optimal transport map from 
$\mathscr L^{d+1}\llcorner \FT(E_{u((k-1)\bar s_n)},\bar s_n)$ to $\mathscr L^{d+1}\llcorner E_{u(k\bar s_n)}$ given by Corollary \ref{C_map}. 
For every $(x,v)\in E_{u_0}$ we build a trajectory $\gamma_{(x,v)}=(\gamma_{(x,v)}^1,\gamma_{(x,v)}^2):[0,T)\to \R^d \times [0,+\infty)$.
First we define inductively $\gamma_{(x,v)}(k\bar s_n)$ for $k=0,\ldots, 2^n-1$. We set
\begin{equation}\label{E_def_gamma1}
\begin{cases}
\gamma_{(x,v)}(0)=(x,v), \\
\gamma_{(x,v)}(k\bar s_n)= T_k( \gamma_{(x,v)}((k-1)\bar s_n) + (\bar s_n f'(\gamma^2_{(x,v)}((k-1)\bar s_n)),0)) \quad \mbox{for }k=1,\ldots, 2^n-1.
\end{cases}
\end{equation}
Next we set for $t\in (k\bar s_n,(k+1)\bar s_n)$ and $k=0,\ldots, 2^n-1$
\begin{equation}\label{E_def_gamma2}
\gamma_{(x,v)}(t)=\gamma_{(x,v)}(k\bar s_n) + ((t-k\bar s_n)f'(\gamma^2_{(x,v)}(k\bar s_n)),0).
\end{equation}

We now define $\omega_n\in \mathcal P(\tilde \Gamma)$ by
\begin{equation}\label{E_def_omega}
\omega_n:=\int_{E_{u_0}}\delta_{\gamma_{(x,v)}}dxdv,
\end{equation}
where $\gamma_{(x,v)}$ is defined by \eqref{E_def_gamma1} and \eqref{E_def_gamma2}.

\begin{lemma}\label{L_hor_vert}
Let $\omega_n$ be defined in \eqref{E_def_omega}. Then the following integral estimates hold:
\begin{enumerate}
\item
\begin{equation}\label{E_est_hor}
\int_{\tilde \Gamma}\sup_{t\in [0,T)}\left|\gamma^1(t)-\gamma^1(0) - \int_0^t f'(\gamma^2(s))ds\right|d\omega_n(\gamma) \le 
\frac{1}{L_n}\left(1+\frac{1}{L_n}\right)\nu((0,T)\times \R^d).
\end{equation}
\item
\begin{equation}\label{E_est_vert}
\int_{\tilde \Gamma}\TV_{[0,T)}\gamma^2 d\omega_n(\gamma) \le \left(1+\frac{1}{L_n}\right)\nu((0,T)\times \R^d).
\end{equation}
\end{enumerate}
\end{lemma}
\begin{proof}
For every $t\in [0,T)$ and $(x,v)\in E_{u_0}$, by the construction of the curves $\gamma_{(x,v)}$, it holds
\begin{equation*}
\begin{split}
\left|\gamma_{(x,v)}^1(t) - \gamma_{(x,v)}^1(0)- \int_0^tf'(\gamma_{(x,v)}^2(s))ds\right| \le &~ 
	\sum_{k=1}^{2^n-1} \left| \gamma_{(x,v)}^1(k\bar s_n) - \gamma_{(x,v)}^1(k\bar s_n-)\right| \\
= &~ \sum_{k=1}^{2^n-1} \left| T_k^1(\gamma_{(x,v)}(k\bar s_n-))- \gamma_{(x,v)}^1(k\bar s_n-)\right|,
\end{split}
\end{equation*}
where we denoted by $ \gamma_{(x,v)}^1(k\bar s_n-):= \lim_{s\to k\bar s_n-}\gamma_{(x,v)}^1(s)$.
Therefore by definition of $\omega_n$ and \eqref{E_est_T} it holds
\begin{equation*}
\begin{split}
\int_{\tilde \Gamma}\sup_{t\in [0,T)} \bigg|\gamma^1(t) -\gamma^1(0) -& \int_0^t f'(\gamma^2(s))ds\bigg|  d\omega_n(\gamma) =  \\
= & \int_{E_{u_0}} \sup_{t \in [0,T)}\left|\gamma_{(x,v)}^1(t) - \gamma_{(x,v)}^1(0)- \int_0^tf'(\gamma_{(x,v)}^2(s))ds\right| dxdv \\
\le &  \sum_{k=1}^{2^n-1} \int_{E_{u_0}}\left| T_k^1(\gamma_{(x,v)}(k\bar s_n-))- \gamma_{(x,v)}^1(k\bar s_n-)\right| dxdv \\
\le &  \sum_{k=1}^{2^n-1} \frac{1}{L_n}\left( 1+\frac{1}{L_n}\right)\nu (((k-1)\bar s_n,k\bar s_n)\times \R^d) \\
\le & \frac{1}{L_n}\left(1+\frac{1}{L_n}\right)\nu((0,T)\times \R^d).
\end{split}
\end{equation*}
This proves \eqref{E_est_hor}  and similarly we get \eqref{E_est_vert}:
\begin{equation*}
\begin{split}
\TV_{[0,T)}\gamma_{(x,v)}^2= &~ \sum_{k=1}^{2^n-1} \left| \gamma_{(x,v)}^2(k\bar s_n) - \gamma_{(x,v)}^2(k\bar s_n-)\right| \\
= &~ \sum_{k=1}^{2^n-1} \left| T_k^2(\gamma_{(x,v)}(k\bar s_n-))- \gamma_{(x,v)}^2(k\bar s_n-)\right|.
\end{split}
\end{equation*}
Integrating this with respect to $\mathscr L^{d+1}\llcorner E_{u_0}$ we get \eqref{E_est_vert} by \eqref{E_est_T}.
\end{proof}

\begin{lemma}\label{L_est_st}
For every $0\le s\le t<T$ it holds
\begin{equation*}
W_1((e_t)_\sharp \omega_n, (e_s)_\sharp \omega_n)\le \|f'\|_{L^\infty}|t-s| + 2\nu\left(\left(s-\frac{T}{2^n},t\right)\times \R^d\right).
\end{equation*}
\end{lemma}
\begin{proof}
Let $k_s,k_t\in [0,2^n-1]\cap \Z$ be such that
\begin{equation*}
k_s\bar s_n \le s<(k_s+1)\bar s_n \qquad \mbox{and} \qquad k_t\bar s_n \le t<(k_t+1)\bar s_n.
\end{equation*}
If $k_s=k_t$ then we have
\begin{equation*}
(e_s)_\sharp \omega_n = \mathscr L^{d+1}\llcorner \FT(E_{u(k_s\bar s_n)},s-k_s\bar s_n)
\end{equation*}
and 
\begin{equation*}
\begin{split}
(e_t)_\sharp \omega_n =&~  \mathscr L^{d+1}\llcorner \FT(E_{u(k_s\bar s_n)},t-k_s\bar s_n) \\
= &~ \mathscr L^{d+1}\llcorner \FT(\FT(E_{u(k_s\bar s_n)},s-k_s\bar s_n),t-s).
\end{split}
\end{equation*}
Therefore the map $T:\R^d\times [0,+\infty)\to \R^d\times [0,+\infty)$ defined by
\begin{equation*}
T(x,v)=(x+f'(v)(t-s),v)
\end{equation*}
satisfies the constraint $T_\sharp((e_s)_\sharp \omega_n)=(e_t)_\sharp \omega_n$ and this proves that
\begin{equation*}
W_1((e_t)_\sharp \omega_n, (e_s)_\sharp \omega_n)\le \|f'\|_{L^\infty}|t-s|.
\end{equation*}
Otherwise it holds $k_s<k_t$ and we estimate by the triangular inequality
\begin{equation*}
\begin{split}
W_1((e_t)_\sharp \omega_n, (e_s)_\sharp \omega_n)\le &~ W_1\left((e_s)_\sharp \omega_n, \lim_{r\to (k_s+1)\bar s_n^-}(e_r)_\sharp \omega_n\right) + \sum_{k=k_s+1}^{k_t}W_1\left(\lim_{r\to k\bar s_n^-}(e_r)_\sharp \omega_n, (e_{k\bar s_n})_\sharp \omega_n \right) \\
& +  \sum_{k=k_s+1}^{k_t-1}W_1\left((e_{k\bar s_n})_\sharp \omega_n, \lim_{r\to (k+1)\bar s_n^-}(e_r)_\sharp \omega_n\right) 
+ W_1\left( (e_{k_t\bar s_n})_\sharp \omega_n, (e_t)_\sharp \omega_n\right) \\
\le &~ \|f'\|_{L^\infty}((k_s+1)\bar s_n-s) + \sum_{k=k_s+1}^{k_t}W_1\left(\lim_{r\to k\bar s_n^-}(e_r)_\sharp \omega_n, (e_{k\bar s_n})_\sharp \omega_n \right) \\
& + (k_t-k_s-1)\|f'\|_{L^\infty}\bar s_n + \|f'\|_{L^\infty}(t-k_t\bar s_n),
\end{split}
\end{equation*}
where the second inequality easily follows by the definition of $\omega_n$ and by the case $k_s=k_t$. Assuming $L_n\ge 1$, which trivially holds for $n$ large enough, we get from \eqref{E_est_T} that for every $k=k_s+1, \ldots, k_t$,
\begin{equation*}
\begin{split}
W_1\left(\lim_{r\to k\bar s_n^-}(e_r)_\sharp \omega_n, (e_{k\bar s_n})_\sharp \omega_n \right) 
= &~ \int_{\FT(E_{u((k-1)\bar s_n)},\bar s_n)}|T_k(x,v)-(x,v)|dxdv\\
\le&~ 2\nu((k-1)\bar s_n, k\bar s_n)\times \R^d).
\end{split}
\end{equation*}
Finally we have
\begin{equation*}
W_1((e_t)_\sharp \omega_n, (e_s)_\sharp \omega_n)\le  \|f'\|_{L^\infty}|t-s| + 2\nu((k_s\bar s_n,k_t\bar s_n)\times \R^d).
\end{equation*}
The conclusion follows since 
\begin{equation*}
(k_s\bar s_n,k_t\bar s_n)\subset \left(s-\frac{T}{2^n},t\right). \qedhere
\end{equation*}
\end{proof}

The following theorem is the main result of this section.
\begin{theorem}\label{T_convergence}
The sequence $\omega_n \in \mathcal P(\tilde \Gamma)$ defined in \eqref{E_def_omega} is tight and any of its limit points is a 
Lagrangian representation of $u$. In particular any weak solution to \eqref{E_multiD} with finite entropy production satisfying Assumption \ref{A_u}
admits a Lagrangian representation.
\end{theorem}
\begin{proof}
Step 1. The sequence $(\omega_n)_{n\in \N}$ is tight in $\mathcal P(\tilde \Gamma)$. 
For every $n\in \N$ and $M,R>0$ we denote by $\tilde \Gamma_{n,M,R}\subset \tilde \Gamma$ the set of curves $\gamma$ such that the following conditions hold:
\begin{enumerate}
\item $\gamma(0)\in B_R(0)\times \left[0,\|u\|_{L\infty}\right]$;
\item for every $k=1,\ldots, 2^n-1$
\begin{equation*}
 \Lip (\gamma\llcorner [(k-1)2^{-n}T,k2^{-n}T))\le \|f'\|_{L^\infty};
\end{equation*}
\item $\TV_{[0,T)}\gamma^2\le M$;
\item 
\begin{equation*}
\sum_{k=1}^{2^n-1}|\gamma^1(2^{-n}Tk) - \gamma^1(2^{-n}Tk-) | \le M2^{-n/2}.
\end{equation*}
\end{enumerate}
Let moreover
\begin{equation*}
\Gamma_{M,R}:= \left\{ \gamma \in \Gamma: \gamma(0)\in B_R\times [0,\|u\|_\infty], \Lip(\gamma^1)\le \|f'\|_{L^\infty}, \TV\gamma^2 \le M \right\}.
\end{equation*}
For every $M,R>0$ the set
\begin{equation*}
\tilde \Gamma_{M,R}:= \Gamma_{M,R}\cup \bigcup_{n=1}^\infty \tilde \Gamma_{n,M,R}
\end{equation*}
is compact in $\tilde \Gamma$ with the topology $\tau$. Moreover it follows from Lemma \ref{L_hor_vert} that for every $\e>0$ there exist $M,R>0$ such that
for every $n\in \N$ it holds $\omega_n(\tilde\Gamma_{M,R}^c)\le \e$.

Step 2. Let $\omega$ be a limit point of the sequence $(\omega_n)_{n\in \N}$. Then \eqref{E_repr_formula} holds true.
We prove separately the two convergences in the sense of distributions:
\begin{enumerate}
\item for every $t\in [0,T)$
\begin{equation}\label{E_repr1}
\lim_{n\to \infty} (e_t)_\sharp \omega_n = \mathscr L^{d+1}\llcorner E_{u(t)};
\end{equation}
\item for every $t\in [0,T)$
\begin{equation}\label{E_repr2}
\lim_{n\to \infty}(e_t)_\sharp \omega_n= (e_t)_\sharp \omega.
\end{equation}
\end{enumerate}
We first notice that \eqref{E_repr1} is trivially true for every $t\in [0,T)$ of the form $t=k T 2^{-N}$ with $k,N\in \N$ since in this case it holds
$(e_t)_\sharp \omega_n = \mathscr L^{d+1}\llcorner E_{u(t)}$ for every $n\ge N$.
We observe that $u$ is continuous in $L^1(\R^d)$ with respect to $t$ by assumption and that since $(\pi_t)_\sharp \nu$ has no atoms
the sequence of curves $t\mapsto (e_t)_\sharp \omega_n$ is converging
uniformly to a continuous limit with respect to the $W_1$ distance by Lemma \ref{L_est_st}.
Therefore \eqref{E_repr1} holds for every $t\in [0,T)$ by continuity.

Let $\omega_{n_k}$ be a weakly convergent subsequence and denote by $\omega$ its limit. 
Notice that $e_t$ is not continuous on $\tilde \Gamma$ endowed with the topology $\tau$ introduced above, so we cannot directly deduce
\eqref{E_repr2} from the weak convergence of $\omega_{n_k}$ to $\omega$. Let $\Delta t>0$ and consider 
$\tilde e_t^{\Delta t}: \tilde \Gamma \to \R^d\times [0,+\infty)$ defined by
\begin{equation*}
\tilde e_t^{\Delta t}(\gamma):=\frac{1}{\Delta t}\int_t^{t+\Delta t}e_s(\gamma) ds.
\end{equation*}
This operator is actually continuous and therefore we get that 
\begin{equation}\label{E_Leb1}
\lim_{k\to \infty}(\tilde e_t^{\Delta t})_\sharp \omega_{n_k} = (\tilde e_t^{\Delta t})_\sharp \omega.
\end{equation}
On the other hand by \eqref{E_repr1} it holds 
\begin{equation}\label{E_Leb2}
\lim_{k\to \infty}(\tilde e_t^{\Delta t})_\sharp \omega_{n_k}= \frac{1}{\Delta t}\int_{t}^{t+\Delta t}\mathscr L^{d+1}\llcorner E_{u(s)}ds.
\end{equation}
From \eqref{E_Leb1} and \eqref{E_Leb2} we get that \eqref{E_repr2} holds for $\mathscr L^1$-a.e. $t\in [0,T)$.
The equality actually holds for every $t\in [0,T)$ since $(e_t)_\sharp \omega$ is continuous with respect to $t$.

Step 3. The measure $\omega$ is concentrated on characteristic curves and \eqref{E_reg} holds true. 
Notice that the function $g:\tilde \Gamma\to \R$ defined by
\begin{equation*}
g(\gamma):=  \sup_{t\in [0,T)}\left|\gamma^1(t)-\gamma^1(0) - \int_0^t f'(\gamma^2(s))ds\right|
\end{equation*}
is lower semicontinuous, therefore
\begin{equation*}
\int_{\tilde\Gamma}g(\gamma) d \omega \le \lim_{n\to \infty}\int_{\tilde\Gamma}g(\gamma) d \omega_n, 
\end{equation*}
which is actually equal to 0 by \eqref{E_est_hor}, since $L_n \to \infty$ as $n\to \infty$.
Similarly \eqref{E_est_vert} implies \eqref{E_reg}.
\end{proof}

In the last part of this section we show how it is possible to decompose the entropy production measures $\mu^\eta$ of any entropy $\eta$ along the
characteristic curves.  

Let $\eta$ be a convex entropy and set
\begin{equation*}
\mu^{\eta}_\gamma=(\Id, \gamma)_\sharp \Big( \big( \eta'' \circ \gamma^2 \big) \tilde D \gamma^2 \Big) + 
\eta''(v) \left(\H^1\llcorner E_{\gamma}^+ -\H^1\llcorner E_{\gamma}^-\right),
\end{equation*}
where
\begin{equation*}
\begin{split}
E_\gamma^+:=&\{(t,x,v): \gamma^1(t)=x, \gamma^2(t-)<\gamma^2(t+), v \in (\gamma^2(t-),\gamma^2(t+)) \}, \\
E_\gamma^-:=&\{(t,x,v): \gamma^1(t)=x, \gamma^2(t+)<\gamma^2(t-), v \in (\gamma^2(t+),\gamma^2(t-)) \}.
\end{split}
\end{equation*}
Accordingly we define
\begin{equation*}
\bar \mu^\eta:=\int_\Gamma \mu^\eta_\gamma \, d\omega.
\end{equation*}

In the next proposition we exploit the relation between the measures $\mu$ in Proposition \ref{P_kin_t} and $\bar \mu^\eta$.

\begin{proposition}\label{P_kin_meas}
Let $\bar \eta(u)=u^2/2$ and $\omega\in \mathcal P_1(\tilde \Gamma)$ as in Theorem \ref{T_convergence}.
Then
\begin{equation}\label{E_kin_lag}
\mu = \int_{\Gamma}\mu^{\bar \eta}_\gamma d\omega(\gamma) \qquad \mbox{and}\qquad 
|\mu| = \int_{\Gamma}\left|\mu^{\bar \eta}_\gamma\right| d\omega(\gamma).
\end{equation}
\end{proposition}
\begin{proof}
Let $\phi \in C^1_c((0,T)\times \R^d\times (0,+\infty))$. Testing \eqref{E_kin_t} with $\bar \phi:=\phi\eta'(v)$ we get
\begin{equation}\label{E_exp_1}
\int\partial_v\bar \phi  d\mu = \int_{E_u}\eta'(v)\left(\partial_t \phi + f'(v)\cdot \nabla_x\phi\right)dtdxdv.
\end{equation}
Being $\omega$ a Lagrangian representation of $u$ it holds
\begin{equation}\label{E_long1}
\begin{split}
 \int_{E_u}\eta'(v)\left(\partial_t \phi + f'(v)\cdot \nabla_x\phi\right)dtdxdv = &~
 \int_\R\int_\Gamma \eta'(\gamma^2(t))\left( \partial_t\phi(t,\gamma(t))+f'(\gamma^2(t))\cdot \nabla_x\phi(t,\gamma(t))\right) d\omega dt \\
 = &~ \int_\Gamma\int_{(0,T)}  \eta'(\gamma^2(t))\left( \partial_t\phi(t,\gamma(t))+\dot{\gamma}^1(t)\cdot \nabla_x\phi(t,\gamma(t))\right) dt d\omega \\
 = &~ \mathrm{I}.
\end{split}
\end{equation}
Set $\phi_\gamma:(0,T)\to \R$ be defined by $\phi_\gamma(t)=\phi(t,\gamma(t))$.
By \eqref{E_reg} for $\omega$-a.e. $\gamma \in \mathcal P(\Gamma)$ the function $\phi_\gamma$ has bounded variation and by the chain rule for $BV$ functions the following equality between measures holds:
\begin{equation*}
\begin{split}
D_t(\eta'\circ\gamma^2\phi_\gamma) = &~ \eta'(\gamma^2(t))\left(\partial_t\phi(t,\gamma(t)) + 
\dot\gamma^1(t)\cdot\nabla_x\phi(t,\gamma(t))\right)\mathscr L^1 + \eta'(\gamma^2(t))\partial_v\phi(t,\gamma(t))\tilde D_t\gamma^2 \\
& + \phi_\gamma(t)\eta''(\gamma^2(t))\tilde D_t\gamma^2 + 
 \sum_{t_j\in J_\gamma}\left(\eta'(\gamma^2(t_j+))\phi_\gamma(t_j+)-\eta'(\gamma^2(t_j-))\phi_\gamma(t_j-)\right)\delta_{t_j},
\end{split}
\end{equation*}
where $J_\gamma$ denotes the jump set of $\gamma$ and $\tilde D_t\gamma^2$ denotes the diffuse part of the measure $D_t\gamma^2$, i.e. the absolutely continuous part plus the Cantor part (see \cite{AFP_book}).
Plugging it into \eqref{E_long1} we get
\begin{equation}\label{E_long2}
\begin{split}
\mathrm I =&~ - \int_\Gamma \int_{(0,T)}\left( \eta'(\gamma^2(t))\partial_v\phi(t,\gamma(t)) + \eta''(\gamma^2(t))\phi(t,\gamma(t))\right) 
d \tilde D_t \gamma^2 (t)d\omega (\gamma) \\
&~ - \int_\Gamma \sum_{t_j\in J_\gamma}\left(\eta'(\gamma^2(t_j+))\phi_\gamma(t_j+)-\eta'(\gamma^2(t_j-))\phi_\gamma(t_j-)\right)d\omega(\gamma) \\
 =&~  - \int_\Gamma \int_{\R^{d+2}} \partial_v \bar \phi d\mu^{\bar \eta}_\gamma d\omega(\gamma).
\end{split}
\end{equation}
Comparing this with \eqref{E_exp_1} we get the first expression in \eqref{E_kin_lag}.

In order to prove the second part of the statement notice that  the following inequality between measures trivially holds:
\begin{equation*}
|\mu| \le \int_{\Gamma}\left|\mu^{\bar \eta}_\gamma\right| d\omega(\gamma).
\end{equation*}
In order to conclude that the equality holds it is actually enough to check that 
\begin{equation*}
\left(\int_{\Gamma}\left|\mu^{\bar \eta}_\gamma\right| d\omega(\gamma)\right)((0,T)\times \R^d\times (0,+\infty)) \le |\mu|((0,T)\times \R^d\times (0,+\infty)).
\end{equation*}
Being $|\mu^{\bar\eta}_\gamma|((0,T)\times \R^d\times (0,+\infty)) =\TV_{(0,T)}\gamma^2$ lower semicontinuous with respecto to $\gamma$
it holds 
\begin{equation*}
\left(\int_{\Gamma}\left|\mu^{\bar \eta}_\gamma\right| d\omega(\gamma)\right)((0,T)\times \R^d\times (0,+\infty)) \le
\liminf_{n\to \infty}\left(\int_{\Gamma}\TV_{(0,T)}(\gamma^2) d\omega_n(\gamma)\right)((0,T)\times \R^d\times (0,+\infty)).
\end{equation*}
By Lemma \ref{L_hor_vert} we finally have that
\begin{equation*}
\begin{split}
\liminf_{n\to \infty}\left(\int_{\Gamma}\TV_{(0,T)}(\gamma^2) d\omega_n(\gamma)\right)((0,T)\times \R^d\times (0,+\infty))\le &~
\nu ((0,T)\times \R^d) \\
= &~ |\mu|((0,T)\times \R^d\times (0,+\infty))
\end{split}
\end{equation*}
and this concludes the proof.
\end{proof}
 Finally we exploit the well-known relation between the measure $\mu$ and the entropy dissipation measures $\mu^\eta$ to decompose them
 along the characteristic curves.
 
 \begin{proposition}\label{P_any_ent}
 For every smooth convex entropy $\eta$ the following representation formula holds:
 \begin{equation*}
 (\pi_{t,x})_\sharp \left( \int_\Gamma \mu^\eta_\gamma d\omega \right) = \mu^\eta.
 \end{equation*}
 \end{proposition}
 \begin{proof}
 Assume without loss of generality that $\eta(0)=0$ and $q(0)=0$. Given $\phi \in  C^1_c((0,T)\times \R^d)$ and by means of the elementary
 identities
 \begin{equation*}
	u(t,x) = \int_0^{+\infty} \chi_{[0,u(t,x)]} (v) \, dv
	\end{equation*}
	and
	\begin{equation*}
	\eta(u(t,x)) = \int_0^{+\infty} \chi_{[0,u(t,x)]} (v)\eta^\prime(v) \, dv, \qquad  q(u(t,x)) = \int_0^{+\infty} \chi_{[0,u(t,x)]}(v) q^\prime(v) \, dv
	\end{equation*}
we easily get
\begin{equation}\label{E_diss1}
- \langle \eta(u)_t + \div_x(\mathbf q(u)), \phi \rangle = \int_{E_u} \eta^\prime(v) \big(  \varphi_t(t,x) + f^\prime(v) \cdot \nabla_x \varphi(t,x) \big) \, dtdxdv.
\end{equation}
 The same computation as in \eqref{E_long1} and \eqref{E_long2} with $\bar \phi(t,x,v)=\eta'(v)\varphi(t,x)$ leads to
 \begin{equation}\label{E_diss2}
 \int_{E_u} \eta^\prime(v) \big(  \varphi_t(t,x) + f^\prime(v) \cdot \nabla_x \varphi(t,x) \big) \, dtdxdv = - \int_\Gamma \int \eta''(v)\phi d\mu^{\bar \eta}_\gamma d\omega.
 \end{equation}
 The conclusion follows from \eqref{E_diss1}, \eqref{E_diss2} and the elementary identity
 \begin{equation*}
 \mu^\eta_\gamma=\eta''(v)\mu^{\bar \eta}_\gamma. \qedhere
 \end{equation*}
  \end{proof}

\begin{remark}
A strictly related statement to Property (3') in the Introduction is the following claim: 
a Lagrangian representation $\omega$ is concentrated on a set of curves $\gamma\in \Gamma$ such that $D_t\gamma^2$ is purely atomic.
Notice that this formulation is natural for general smooth fluxes, even without any nonlinearity assumption.
As already mentioned in the introduction this claim has been proved in several space dimension only for continuous entropy solutions in \cite{BBM_multid},
where actually $D_t\gamma^2=0$ for $\omega$-a.e. $\gamma \in\Gamma$.
\end{remark}

\bibliographystyle{alpha}

\begin{thebibliography}{DLOW03}

\bibitem[ADL04]{ADL_note}
L.~Ambrosio and C.~De~Lellis.
\newblock A note on admissible solutions of 1{D} scalar conservation laws and
  2{D} {H}amilton-{J}acobi equations.
\newblock {\em J. Hyperbolic Differ. Equ.}, 1(4):813--826, 2004.

\bibitem[AFP00]{AFP_book}
L.~Ambrosio, N.~Fusco, and D.~Pallara.
\newblock {\em Functions of Bounded Variation and Free Discontinuity Problems}.
\newblock Oxford Science Publications. Clarendon Press, 2000.

\bibitem[Amb04]{Ambrosio_transport}
Luigi Ambrosio.
\newblock Transport equation and {C}auchy problem for {$BV$} vector fields.
\newblock {\em Invent. Math.}, 158(2):227--260, 2004.

\bibitem[BBM17]{BBM_multid}
S.~Bianchini, P.~Bonicatto, and E.~Marconi.
\newblock A lagrangian approach to multidimensional conservation laws.
\newblock {\em preprint SISSA 36/MATE}, 2017.

\bibitem[BBMN10]{bertini_al}
G.~Bellettini, L.~Bertini, M.~Mariani, and M.~Novaga.
\newblock $\Gamma$-entropy cost for scalar conservation laws.
\newblock {\em Archive for Rational Mechanics and Analysis}, 195:261--309, 2010.

\bibitem[BD18]{BD_L1map}
Stefano Bianchini and Sara Daneri.
\newblock On {S}udakov's type decomposition of transference plans with norm
  costs.
\newblock {\em Mem. Amer. Math. Soc.}, 251(1197):vi+112, 2018.

\bibitem[BM17]{BM_structure}
S.~Bianchini and E.~Marconi.
\newblock On the structure of $l^\infty$ entropy solutions to scalar
  conservation laws in one-space dimension-entropy solutions to scalar
  conservation laws in one-space dimension.
\newblock {\em Archive for Rational Mechanics and Analysis}, Jun 2017.

\bibitem[Bre84]{Brenier_TC}
Y.~Brenier.
\newblock Averaged multivalued solutions for scalar conservation laws.
\newblock {\em SIAM J. Numer. Anal.}, 21(6):1013--1037, 1984.

\bibitem[Che86]{Cheng_speed_BV}
K.~S. Cheng.
\newblock A regularity theorem for a nonconvex scalar conservation law.
\newblock {\em J. Differential Equations}, 61(1):79--127, 1986.

\bibitem[COW08]{COW_bologna}
G.~Crippa, F.~Otto, and M.~Westdickenberg.
\newblock {\em Regularizing Effect of Nonlinearity in Multidimensional Scalar
  Conservation Laws}, volume~5 of {\em Lecture Notes of the Unione Matematica
  Italiana}.
\newblock Springer-Verlag, Berlin; UMI, Bologna, 2008.

\bibitem[Daf85]{Dafermos_inflection}
C.~M. Dafermos.
\newblock Regularity and large time behaviour of solutions of a conservation
  law without convexity.
\newblock {\em Proc. Roy. Soc. Edinburgh Sect. A}, 99(3-4):201--239, 1985.

\bibitem[Daf06]{Dafermos_continuous}
C.~M. Dafermos.
\newblock Continuous solutions for balance laws.
\newblock {\em Ric. Mat.}, 55(1):79--91, 2006.

\bibitem[Daf16]{dafermos_book_4}
C.~M. Dafermos.
\newblock {\em Hyperbolic conservation laws in continuum physics}, volume 325
  of {\em Grundlehren der Mathematischen Wissenschaften [Fundamental Principles
  of Mathematical Sciences]}.
\newblock Springer-Verlag, Berlin, fourth edition, 2016.

\bibitem[DLOW03]{DLW_structure}
C.~De~Lellis, F.~Otto, and M.~Westdickenberg.
\newblock Structure of entropy solutions for multi-dimensional scalar
  conservation laws.
\newblock {\em Arch. Ration. Mech. Anal.}, 170(2):137--184, 2003.

\bibitem[DLR03]{DLR_dissipation}
C.~De~Lellis and T.~Rivi{{\`e}}re.
\newblock The rectifiability of entropy measures in one space dimension.
\newblock {\em J. Math. Pures Appl. (9)}, 82(10):1343--1367, 2003.

\bibitem[GL19]{GL_kinetic_entropic}
Benjamin Gess and Xavier Lamy.
\newblock Regularity of solutions to scalar conservation laws with a force.
\newblock {\em Ann. Inst. H. Poincar\'{e} Anal. Non Lin\'{e}aire},
  36(2):505--521, 2019.

\bibitem[GP13]{GP_optimal}
F.~Golse and B.~Perthame.
\newblock Optimal regularizing effect for scalar conservation laws.
\newblock {\em Rev. Mat. Iberoam.}, 29(4):1477--1504, 2013.

\bibitem[Kru70]{Kruzhkov_contraction}
S.~N. Kru{\v{z}}kov.
\newblock First order quasilinear equations with several independent variables.
\newblock {\em Mat. Sb. (N.S.)}, 81 (123):228--255, 1970.

\bibitem[LO18]{LO_Burgers}
Xavier Lamy and Felix Otto.
\newblock On the regularity of weak solutions to {B}urgers' equation with
  finite entropy production.
\newblock {\em Calc. Var. Partial Differential Equations}, 57(4):Art. 94, 19,
  2018.

\bibitem[LPT94]{LPT_kinetic}
P.-L. Lions, B.~Perthame, and E.~Tadmor.
\newblock A kinetic formulation of multidimensional scalar conservation laws
  and related equations.
\newblock {\em J. Amer. Math. Soc.}, 7(1):169--191, 1994.

\bibitem[Mar10]{Mariani_large}
Mauro Mariani.
\newblock Large deviations principles for stochastic scalar conservation laws.
\newblock {\em Probab. Theory Related Fields}, 147(3-4):607--648, 2010.

\bibitem[Mar18]{m_regularity}
Elio Marconi.
\newblock Regularity estimates for scalar conservation laws in one space
  dimension.
\newblock {\em J. Hyperbolic Differ. Equ.}, 15(4):623--691, 2018.

\bibitem[Ole63]{Oleinik_translation}
O.~A. Ole{\u\i}nik.
\newblock Discontinuous solutions of non-linear differential equations.
\newblock {\em Amer. Math. Soc. Transl. (2)}, 26:95--172, 1963.

\bibitem[Sil19]{Silvestre_CL}
Luis Silvestre.
\newblock Oscillation properties of scalar conservation laws.
\newblock {\em Comm. Pure Appl. Math.}, 72(6):1321--1348, 2019.

\bibitem[TT07]{TT_kinetic}
Eitan Tadmor and Terence Tao.
\newblock Velocity averaging, kinetic formulations, and regularizing effects in
  quasi-linear {PDE}s.
\newblock {\em Comm. Pure Appl. Math.}, 60(10):1488--1521, 2007.

\bibitem[Vil09]{V_oldnew}
C\'{e}dric Villani.
\newblock {\em Optimal transport}, volume 338 of {\em Grundlehren der
  Mathematischen Wissenschaften [Fundamental Principles of Mathematical
  Sciences]}.
\newblock Springer-Verlag, Berlin, 2009.
\newblock Old and new.

\end{thebibliography}

\end{document}